\documentclass[11pt,reqno]{amsart} 

\usepackage[a4paper, layout=a4paper, marginratio={1:1, 2:3}]{geometry}
\usepackage[utf8]{inputenc}
\usepackage[T1]{fontenc}
\usepackage{lmodern}
\usepackage{amsmath,amssymb}
\usepackage{amsthm}
\usepackage{paralist}
\usepackage{float}
\usepackage{accents}
\usepackage{color}
\usepackage[authoryear]{natbib}
\usepackage{bm}
\usepackage{booktabs}
\usepackage{tabularx}
\usepackage{lscape}
\usepackage{subfig}
\usepackage{graphicx}
\numberwithin{equation}{section}
\usepackage{dsfont}
\usepackage[colorlinks=true,linkcolor=blue,citecolor=blue,pdfborder={0 0 0}]{hyperref}

\newcommand{\id}{\mathds{1}}

\renewcommand{\epsilon}{\varepsilon}
\newcommand{\eps}{{\varepsilon}}
\renewcommand{\phi}{\varphi}
\DeclareMathOperator{\cum}{cum} 
 
\newcommand{\ubar}[1]{\underaccent{\bar}{#1}}

\newcommand{\R}{\mathbb{R}}
\newcommand{\Z}{\mathbb{Z}}
\newcommand{\N}{\mathbb{N}}

\newcommand{\pr}{\mathbb{P}}        
\newcommand{\Prob}{\mathbb{P}}        
\newcommand{\ex}{\mathbb{E}}        

\newcommand{\Exp}{\mathbb{E}}        
\newcommand{\var}{\textnormal{Var}} 
\newcommand{\cov}{\textnormal{Cov}} 

\newcommand{\vertiii}[1]{{\left\vert\kern-0.25ex\left\vert\kern-0.25ex\left\vert #1 
    \right\vert\kern-0.25ex\right\vert\kern-0.25ex\right\vert}}

\newcommand{\Ac}{\mathcal{A}}
\newcommand{\Bc}{\mathcal{B}}

\newcommand{\Lc}{\mathcal{L}}

\newcommand{\Oc}{\mathcal{O}}

\newcommand{\Sc}{\mathcal{S}}

\newcommand{\Bb}{\mathbb{B}}

\newcommand{\rel}{\textnormal{rel}}

\newcommand{\diff}{{\,\mathrm{d}}}

\newcommand{\scs}{\scriptscriptstyle}

\newcommand{\convd}{\rightsquigarrow}              
\newcommand{\convw}{\convd}                           

\newtheorem{theorem}{Theorem}[section]
\newtheorem{proposition}[theorem]{Proposition}

\newtheorem{corollary}[theorem]{Corollary}

\theoremstyle{definition}

\newtheorem{condition}[theorem]{Condition}

\theoremstyle{remark}

\definecolor{gray}{gray}{0.6}

\allowdisplaybreaks[4]

\begin{document}

\title[Detecting serial correlation in LSFTS]{A Portmanteau-type test for detecting serial correlation in locally stationary functional time series}

\author{Axel Bücher}
\thanks{\textit{Corresponding author:} Axel Bücher.}

\author{Holger Dette}

\author{Florian Heinrichs}

\address[Axel Bücher]{Heinrich-Heine-Universit\"at D\"usseldorf, Mathematisches Institut, Universit\"atsstr.~1, 40225 D\"usseldorf, Germany.}
\email{axel.buecher@hhu.de}
\address[Holger Dette, Florian Heinrichs]{Ruhr-Universit\"at Bochum, Fakult\"at f\"ur Mathematik, Universit\"atsstr.\ 150, 44780 Bochum, Germany.}
\email{holger.dette@rub.de, florian.heinrichs@rub.de}

\begin{abstract} 
The Portmanteau test provides the vanilla method for detecting serial correlations in classical univariate time series analysis. The method is extended to the case of observations from a locally stationary functional time series. Asymptotic critical values are obtained by a suitable block multiplier bootstrap procedure. The test is shown to asymptotically hold its level and to be consistent against general alternatives.

\medskip

\end{abstract}

\date{\today}

\maketitle


\section{Introduction} \label{sec:intro}

Over the last decades, technological progress allowed to store more and more data. In particular, many time series are recorded with a very high frequency, as for instance intraday prices of stocks or temperature records. In the literature, data of this type is often viewed as functional observations. Due to this development, the field of functional data analysis has been very active recently (see the monographs \citealp{Bos00},  \citealp{ferrvieu2006}, \citealp{HorKok12} and \citealp{hsingeubank2015}, among others).

The statistical analysis of functional data simplifies substantially if the observations are serially uncorrelated (or even serially independent). In fact, a huge amount of methodology has been proposed solely for this scenario, whence it is important to validate or reject this assumption in applications. Moreover, in the context of (univariate) financial return data, the absence or insignificance of serial correlation is commonly interpreted as a sign for efficient market prices \citep{fama70}. Likewise, investors may be interested in knowing whether functional counterparts like cumulative intraday returns exhibit significant autocorrelation. 

If the observations are not only serially uncorrelated, but also centred and  homoscedastic, then the time series is referred to as a functional white noise. Testing for the functional white noise hypothesis has found considerable interest in the recent literature.  
For instance, inspired by classical portmanteau-type methodology  in univariate or multivariate time series analysis \citep[see][among others]{BoxPierce1970,Hosking1980,Hong1996,Pena2002}, \cite{gabrys2007} propose to apply a multivariate portmanteau test to vectors of a few principal components
from the functional time series.
\cite{horvathetal2013} and  \cite{kokoszka2017} investigate a portmanteau-type test which is  based on estimates of the norms of the autocovariance operators.
Alternatively,  tests in the frequency domain  have been proposed as well, which are based on the fact that the spectral density operator of a functional white noise time series is constant. 
\cite{zhang2016} proposes a Cramér-von Mises type test based on the functional periodogram and  
\cite{bagchi2018}  suggest a test based on an estimate of the minimum distance between  the spectral density operator and  its best approximation by a constant. 
While this approach estimates the minimal distance directly avoiding  estimation of  the spectral density operator,   \cite{characiejus2020} 
suggest a test which is based  on the distance between  the  estimated spectral density operator  and  an estimator of the operator 
calculated under the assumption  of an uncorrelated time series.

A common feature of all aforementioned references consists in the fact that the proposed methodology is only applicable under the assumption of a (second order) stationary time series.
This  paper goes a step further and investigates the problem of testing for uncorrelatedness in possibly non-stationary functional data; in particular, for certain forms of heteroscedasticity. More precisely, we propose a portmanteau type test for locally stationary functional time series, whose critical values may be obtained by a multiplier block bootstrap. As a by-product, if accompanied by a test for constancy of the variance (see, e.g., \citealp{Buecher2020}), we straightforwardly obtain a test for the null hypothesis of functional white noise as well.
Finally, we propose a generalized procedure to test for so-called \textit{relevant} serial correlations, see  Section \ref{sec:relHyp} for a rigorous definition.

 The paper is organized as follows: mathematical preliminaries, including a precise description of the hypotheses, are collected in Section~\ref{sec:math}. Suitable test statistics are introduced in Section~\ref{sec:main}, where we also prove weak convergence and validate a bootstrap approximation to obtain suitable critical values. Finite sample results are collected in Section~\ref{sec:sim}, a case study is presented in Section~\ref{sec:case} and all proofs are deferred to Section~\ref{sec:proofs}.

	\section{Mathematical Preliminaries} \label{sec:math}

Throughout this document, we deal with objects in $L^p([0,1]^d)$, for different choices of $p\geq 1$ and $d\in\N$. We denote the respective $L^p$-norms by $\|\cdot\|_{p,d}$, with the special case $\|\cdot\|_{p,1}$ abbreviated by $\|\cdot\|_p$. Further, for functions $f,g\in L^p([0,1])$, we write $(f\otimes g)(x,y)=f(x)g(y)$.

\subsection{Locally stationary time series}

For $t\in\Z$, let $X_t:[0,1]\times \Omega\to\R$ denote a  $(\Bc|_{[0,1]}  \otimes \Ac)$-measurable function with $X_t(\cdot,\omega)\in \Lc^2([0,1])$ for any $\omega \in \Omega$. We can regard $[X_t]$ as a random variable in $L^2([0,1])$ and will denote it by $X_t$ as well. The expected value of $[X_t]$ in $L^2([0,1])$ coincides with the equivalence class of $ \tau \mapsto \mu_t(\tau)=\ex[X_t(\tau)]$. Similarly, the kernel of the (auto-)covariance operator of $[X_t]$ has a representation in $\Lc^2([0,1]^2)$ with $c_{X_t}(\tau,\sigma)=\cov\big(X_t(\tau),X_t(\sigma)\big)$ and $c_{X_t,X_{t+h}}(\tau,\sigma)=\cov\big(X_t(\tau),X_{t+h}(\sigma)\big)$. We refer to Section 2.1 in \cite{Buecher2020} for technical details.

The sequence $(X_t)_{t\in\Z}$ is called a \textit{functional time series} in $L^2([0,1])$. The sequence is called \textit{stationary} if, for all $q\in\N$ and $h,t_1,\dots,t_q\in\Z$,
$$ (X_{t_1+h},\dots,X_{t_q+h})\overset{d}{=} (X_{t_1},\dots,X_{t_q}) $$
in $L^2([0,1])^q$.
For the definition of a locally stationary functional time series we use a concept introduced by \cite{Vogt2012} and \cite{vanDelft2018} \citep[see also][] {VanChaDet17,delftdette2018, Buecher2020}. To be precise we call a sequence of functional time series $(X_{t,T})_{t\in\Z}$ indexed by $T\in\N$ a \textit{locally stationary functional time series of order $\rho>0$} if there exists, for any $u\in[0,1]$, a stationary functional time series $(X_t^{\scs (u)})_{t\in\Z}$ in $L^2 ([0,1])$ and an array of real-valued random variables $\{P_{t,T}^{(u)}:t=1,\dots,T,T\in\N\}$ with $\ex|P_{t,T}^{\scs (u)}|^\rho<\infty$ uniformly in $t\in\{1,\dots,T\}, T\in\N$ and $u\in[0,1]$, such that
$$ \|X_{t,T}-X_t^{(u)}\|_2 = \bigg\{ \int_{0}^1 \{X_{t,T}(\tau) - X_{t}^{(u)}(\tau)\}^2 \diff \tau\bigg\}^{1/2}
 \leq
 \bigg(\bigg|\frac{t}{T}-u\bigg|+\frac{1}{T}\bigg)P_{t,T}^{(u)}, $$
for any $t\in\{1,\dots,T\}, T\in\N$ and $u\in[0,1]$. Note that in the case  $\rho\ge 1$ the approximating family $\{ (X_t^{\scs (u)})_{t\in\Z} : u\in[0,1]\}$ is $L^2$-Lipschitz continuous in the sense that 
\begin{equation}\label{eq:lipschitz}
\ex \|X_t^{(u)}-X_t^{(v)}\|_2\leq C|u-v|,
\end{equation}
for some constant $C>0$, by local stationarity of $X_{t,T}$. In the following discussion we assume  that 
 $X_{t,T}$ (and hence $X_{t}^{\scs (u)}$)  is centred, i.\,e. $ \mu_{t,T}= \mathbb{E}[X_{t,T}] =0$ for all $t \in \{1,\ldots , T\} $.

\subsection{Serial correlation in locally stationary time series}

In classical (functional) time series analysis, a time series is called uncorrelated if its autocovariances are zero for any lag $h>0$. In the locally stationary setup, a slightly more subtle version suggests itself: 
we call a centred locally stationary functional time series of order $\rho\ge 2$  with approximating family of square-integrable stationary time series $\{ (X_t^{\scs (u)})_{t\in\Z} : u\in[0,1]\}$ (i.\,e., $\Exp[\| X_0^{\scs(u)}\|_2^2]<\infty$ for all $u$)  serially uncorrelated
 if the hypothesis
 \begin{align} \label{hd1}
 \bar H_0 :=H_0^{(1)} \cap H_0^{(2)} \cap \dots
 \end{align}
holds, where the  individual hypothesis $H_0^{\scs (h)}$ at lag $h \in \N$ is defined by  
\begin{align} \label{eq:h01}
H_0^{(h)}:~ \|\cov(X_{0}^{(u)},X_{h}^{(u)})\|_{2,2}
=  0 \hspace{.6cm}\text{for all}~u\in[0,1].
\end{align}
If, additionally, $u \mapsto \var(X_0^{\scs(u)})$ is constant, then the locally stationary time series will be called functional white noise.
 As in Remark 1 in \cite{Buecher2020}, it may be shown that these definitions are independent of the choice of the approximating family. 

Throughout this paper, we will develop suitable tests for certain hypotheses related  to $\bar H_0$ and  $H_0^{\scs (h)}$ in \eqref{hd1} and \eqref{eq:h01}, respectively.
Following the main principle of classical portmanteau-type tests for detecting serial correlations, we start by fixing a maximum lag $H \in \N$ and to test the hypotheses
\begin{align}\label{eq:h02}
\bar H_0^{(H)}: \|\cov(X_{0}^{(u)},X_{h}^{(u)})\|_{2,2} = 0 \hspace{0.6cm} \text{for all}~h\in\{1,\dots,H\}~\text{and}~u\in[0,1] .
\end{align}
Note that  $\bar H_0= \bigcap_{H\in\N} \bar H_0^{\scs (H)}$.

\subsection{Regularity conditions on the observation scheme}
In order to obtain meaningful asymptotic results,  the following regularity conditions  will be imposed.

\begin{condition}[Assumptions on the observations] \label{cond:fts}~
	\begin{enumerate}
		\renewcommand{\theenumi}{(A1)}
		\renewcommand{\labelenumi}{\theenumi}
		\item\label{cond:ls} \textbf{Local Stationarity.} The observations $X_{1,T}, \dots X_{T,T}$ are an excerpt from a centered  locally stationary functional time series $\{(X_{t,T})_{t\in\Z}:T\in\N\}$ of order $\rho=4$ in $L^2([0,1],\R)$, with approximating family of stationary time series $\{ (X_t^{\scs (u)})_{t\in\Z}: u\in[0,1]\}$.
		\renewcommand{\theenumi}{(A2)}
		\renewcommand{\labelenumi}{\theenumi}
		\item\label{cond:mom}  \textbf{Moment Condition.}  
		For any $k\in\N$, there exists a constant $C_k<\infty$ such that $\ex \|X_{t,T}\|_2^k\leq C_k$ and $\ex\|X_0^{\scs (u)}\|_2^k\leq C_k$ uniformly in $t\in\Z,T\in\N$ and $u\in[0,1]$.
		\renewcommand{\theenumi}{(A3)}
		\renewcommand{\labelenumi}{\theenumi}
		
		\item\label{cond:cum}  \textbf{Cumulant Condition.} For any $j\in\N$ there is a constant $D_j <\infty$ such that
		\begin{equation*} 
		\sum_{t_1,\dots ,t_{j-1}=-\infty}^{\infty} \big\| \cum(X_{t_1,T},\dots ,X_{t_j,T})\big\|_{2,j} \leq D_j<\infty, 
		\end{equation*}
		for any $t_j\in\Z$ (for $j=1$ the condition is to be interpreted as $\|\ex X_{t_1,T}\|_2\leq D_1$ for all $t_1 \in \Z$). Further, for $k\in\{2,3,4\}$, there exist functions $\eta_k:\Z^{k-1}\to\R$ satisfying
		\[
		\sum_{t_1,\dots,t_{k-1}=-\infty}^{\infty} (1+|t_1|+\dots+|t_{k-1}|)\eta_k(t_1,\dots,t_{k-1})  < \infty
		\]
		such that, for any $T\in\N, 1 \le t_1 , \dots , t_k \le T, v, u_1, \dots , u_k \in[0,1],h_1,h_2\in\Z$, $Z_{t,T}^{\scs (u)}\in\{X_{\scs t, T},X_{t}^{\scs (u)}\}$,  and any $Y_{t,h,T}(\tau_1,\tau_2)\in\{ X_{t,T}(\tau_1), X_{t,T}(\tau_1)X_{t+h,T}(\tau_2) \}$, we have
		\begin{enumerate}
			\item[(i)] {\small$\|\cum(X_{t_1,T}-X_{t_1}^{(t_1/T)},Z_{t_2,T}^{(u_2)},\cdots,Z_{t_k,T}^{(u_k)})\|_{2,k} \leq \frac{1}{T} \eta_k(t_2-t_1,\dots,t_k-t_1)$},
			\item[(ii)] {\small$\|\cum(X_{t_1}^{(u_1)}-X_{t_1}^{(v)},Z_{t_2,T}^{(u_2)},\cdots,Z_{t_k,T}^{(u_k)})\|_{2,k} \leq |u_1-v| \eta_k(t_2-t_1,\dots,t_k-t_1)$}, 
			\item[(iii)] {\small$\|\cum(X_{t_1,T},\dots,X_{t_k,T})\|_{2,k} \leq \eta_k(t_2-t_1,\cdots,t_k-t_1)$}, 	
			\item[(iv)] {\small$\int_{[0,1]^{2}} |\cum\big(Y_{t_1,h_1,T}(\tau),Y_{t_2,h_2,T}(\tau) \big)|\diff\tau
				\leq \eta_2(t_2-t_1)$.}
		\end{enumerate}	
	\end{enumerate}
\end{condition}

Assumption \ref{cond:ls} restricts the non-stationary behaviour of the observations to smooth changes, while the moment condition ensures existence of the cumulants. The cumulant condition originates from classical multivariate time series analysis \citep[see, e.\,g., ][]{Brillinger1965}. Similar assumptions were made by \cite{LeeSubbaRao2016} and \cite{AueVan2020} in the context of non-stationary  functional data. 
Lemma 2 in \cite{Buecher2020} shows that \ref{cond:cum} follows from \ref{cond:ls}, \ref{cond:mom} and an additional moment condition, provided that a certain strong mixing condition is met.

	\section{Testing for serial correlation in locally stationary functional data} \label{sec:main}

\subsection{A test statistic for detecting serial correlation} \label{sec:clasHyp}

In this section, we propose a test statistic for detecting deviations from hypothesis \eqref{eq:h02} and 
prove  a corresponding weak convergence result.
 A bootstrap device for deriving suitable critical values will be discussed in the subsequent Section~\ref{sec:bootstrap}.

The test statistic is based on the following observation:  as $X_t^{\scs (u)}$  is centred  we may rewrite  (observing \eqref{eq:lipschitz}) 
hypotheses \eqref{eq:h01}  and \eqref{eq:h02} as 
\[
H_0^{(h)}: \|M_h\|_{2,3}=0 \quad \text{ and } \quad  \bar H_0^{(H)} : \max_{h=1}^{H}\|M_h\|_{2,3}=0,
\]
where
$$  M_h(u,\tau_1,\tau_2)=\int_0^u \ex[X_0^{(w)}(\tau_1)X_h^{(w)}(\tau_2)]\diff w.$$
An empirical version of $M_h$, based on the observations $X_{1,T}, \dots, X_{T,T}$, is provided by the statistic
\[
\hat{M}_{h,T}(u,\tau_1,\tau_2)=\frac{1}{T}\sum_{t=1}^{\lfloor uT\rfloor \wedge (T-h)}X_{t,T}(\tau_1)X_{t+h,T}(\tau_2).
\]
The next theorem implies consistency of the empirical versions, which suggests to reject the null hypotheses in  \eqref{eq:h01}  and \eqref{eq:h02}  
for large values of 
the statistics
\[
\Sc_{h,T} = \sqrt T \| \hat M_{h,T}\|_{2,3}  \quad \text{ and } \quad 
\bar \Sc_{H,T} = \sqrt T \max_{h=1}^H \| \hat M_{h,T}\|_{2,3},
\]
respectively.

\begin{theorem}\label{ConvM}
	Under Condition \ref{cond:fts}, we have, for any $h\in\N$ as $T \to \infty$
	$$
{1 \over \sqrt T} \Sc_{h,T}	  \to  \| M_h\|_{2,3} 
$$
in probability. Moreover,  for any $H\in\N$, $ h \in \{ 1, \ldots , H\}$ as $T \to \infty$

\[
\Sc_{h,T}\convw  
\begin{cases} 
\| \tilde B_h\|_{2,3} & \text{ under } H_0^{(h)}, \\
+\infty & \text{ else}, \end{cases}
\quad \text{ and } \quad 
\bar \Sc_{H,T}\convw  
\begin{cases} \max\limits_{h=1}^H \| \tilde B_h\|_{2,3} & \text{ under } \bar H_0^{(H)}, \\
+\infty & \text{ else}, \end{cases} 
\]
where $\tilde{B}=(\tilde B_1,\dots,\tilde B_H)$ denotes a centred Gaussian variable in $L^2([0,1]^3)^H$, whose covariance operator  $C_\Bb:L^2([0,1]^3)^H \to L^2([0,1]^3)^H$ is defined  by
\begin{equation}\label{eq:covOperator}
C_{\Bb} 
\left(\begin{array}{c}
f_1 \\ \vdots \\  f_{H}
\end{array} \right)
\left(\begin{array}{c}
(u_1, \tau_{11}, \tau_{12})  \\ \vdots \\  (u_H, \tau_{H1}, \tau_{H2})
\end{array} \right)
=
\left(\begin{array}{c}
\sum_{h=1}^H \langle r_{1,h} ((u_1,\tau_{11}, \tau_{12}) , \cdot), f_h \rangle  \\
\vdots \\  
\sum_{h=1}^H \langle r_{H,h} ((u_H,\tau_{H1}, \tau_{H2}) , \cdot), f_h \rangle 
\end{array} \right).
\end{equation}
Here, the kernel function $r_{h,h'}$is given by
\begin{align} \nonumber 
r_{h,h'}((u,\tau_1, \tau_2), (v,\phi_1, \phi_2)) 
&= 
\cov\big(\tilde {B}_h(u,\tau_1,\tau_2),\tilde {B}_{h'}(v,\phi_1,\phi_2)\big) \\
&=
\sum_{k=-\infty}^{\infty}\int_{0}^{u\wedge v}c_{k}(w) \diff w,
\label{hd0}
\end{align}
with
\begin{align*}
c_{k}(w)
&=
c_{k}(w, h, h',\tau_1, \tau_2, \varphi_1, \varphi_2)
=
\cov\big(X_0^{(w)}(\tau_1)X_h^{(w)}(\tau_2),X_k^{(w)}(\phi_1)X_{k+h'}^{(w)}(\phi_2)\big),
\end{align*}
for any $1\leq h,h'\leq H$. In particular, the infinite sum in \eqref{hd0}  converges.
\end{theorem}

It is worthwhile to mention that the distributions of the limiting variables in the previous theorems are not pivotal under the null hypotheses. As a consequence, critical values for respective statistical tests must be estimated, for instance by a plug-in approach or by a suitable bootstrap device. Throughout this paper, we 
propose a bootstrap approach which will be worked out in Section \ref{sec:bootstrap} below.

\subsection{Detecting relevant serial correlations} \label{sec:relHyp}
In the previous section, we considered ``classical'' hypotheses in the sense that we were testing  whether the 
 covariance operators up to lag $H$  are exactly equal to zero. 
 However, in concrete applications,  hypotheses of this type might rarely be satisfied exactly and it might rather be reasonable to reformulate the null hypothesis in the form that ``the norm of the autocovariance operator is small'', but not exactly equal to $0$.
More precisely, given thresholds  $\Delta_{h} > 0$
which may vary with the lag  $h  \in  \{1, \ldots , H\} $,  we propose to consider
the following \textit{relevant hypotheses}
\begin{align}
\nonumber
H_0^{(h,\Delta)}&: \|M_h\|_{2,3} \leq \Delta_h,\\
\bar H_0^{(H,\Delta)}&: \|M_h\|_{2,3}  \leq \Delta_h \quad \text{for all}~h\in\{1,\dots,H\},
\label{hd3}
\end{align} 
where $H\in\N$ is some fixed constant representing the maximal lag under consideration. 
The choice of the thresholds  $\Delta_h$ depends on the specific application and has to be discussed with the practitioner in concrete applications. Although this
may be a daunting task, we strongly argue that one should carefully think about it as the classical implicit choice of $\Delta_h=0$ typically corresponds to an unrealistic null hypothesis in many applications.

Consistency of $\hat M_{h,T}$ for $M_h$ suggests to reject the above hypotheses for large values of $\hat M_{h,T}$. We propose to consider the ``normalized''
test statistics 
\begin{align*}
\mathcal S_{h,\Delta_h,T}
&= \sqrt{T}(\|\hat{M}_{h,T}\|_{2,3}-\Delta_h)\|\hat{M}_{h,T}\|_{2,3} ~, \\
\bar{\mathcal S}_{H, \Delta, T} 
&= \max_{h=1}^H  \sqrt{T}(\|\hat{M}_{h,T}\|_{2,3}-\Delta_h)\|\hat{M}_{h,T}\|_{2,3}   ~, 
\end{align*}
whose asymptotic properties are described in the following result. It is worthwhile to mention that related test statistics like $\sqrt T(\| \hat M_{h,T}\|_{2,3}^2-\Delta_h^2)$ or $\sqrt{T}(\|\hat{M}_{h,T}\|_{2,3}-\Delta_h)\Delta_h$ may be treated similarly, but that the respective tests exhibited  a worse finite-sample performance in an unreported Monte-Carlo simulation study.

\begin{corollary}\label{ConvM2}
Under Condition \ref{cond:fts}, we have, for any fixed $H\in\N$ and for $T\to\infty$,
\[ 
\sqrt{T}\big((\|\hat{M}_{h,T}\|_{2,3}-\|M_h\|_{2,3})\|\hat{M}_{h,T}\|_{2,3}\big)_{h=1,\dots,H}\convw\big(\langle M_h,\tilde{B}_h\rangle\big)_{h=1,\dots,H},
\]
where $\tilde B_1, \dots, \tilde B_H$ are defined in Theorem~\ref{ConvM} and
 $\langle f,g \rangle = \int_{[0,1]^3} f(x) g(x) \diff x$. As a consequence, 
 \[
\mathcal S_{h,\Delta_h,T}
\convw\left\{
\begin{array}{rl}
- \Delta_h \| \tilde B_h \|_{2,3} & \text{if } \|M_h\|_{2,3} = 0, \\
-\infty & \text{if } \|M_h\|_{2,3} \in (0,\Delta_h),\\
 \langle M_h, \tilde{B}_h\rangle & \text{if }\|M_h\|_{2,3}=\Delta_h,\\
+\infty & \text{if }\|M_h\|_{2,3}>\Delta_h.
\end{array}\right. 
\]
Moreover, 
\[
\bar {\mathcal S}_{H, \Delta,T} 
\convw\left\{
\begin{array}{rl}
\max\{ \max_{h\in N_H}  \langle M_h, \tilde{B}_h\rangle, \max_{h \in O_h} - \Delta_h \|\tilde B_h\|_{2,3}\} & \text{ if }\bar H_0^{(H,\Delta)} \text{ is met}, \\
+\infty & \text{ else},
\end{array}\right.
\]
where $N_H=\{h \in \{1, \dots, H\}: \|M_h\|_{2,3} = \Delta_h\}$, $O_H=\{h\in\{1, \dots, H\}: \|M_h\|_{2,3}= 0\}$ and where the maximum over the empty set is interpreted as $-\infty$. 
\end{corollary}

As in Section \ref{sec:clasHyp},  the limiting distributions under the null hypotheses are not pivotal, whence a bootstrap procedure will be introduced next.

\subsection{Critical values based on bootstrap approximations}
\label{sec:bootstrap}
The limiting distributions of the  test statistics  derived in the previous sections depend in a complicated way on the higher order serial dependence of the underlying approximating family $\{ (X_t^{\scs (u)})_{t\in\Z}: u\in[0,1]\}$ and are rather difficult to estimate. 
To avoid the estimation, we propose a multiplier block bootstrap procedure.

Following  \cite{Buecher2020} the bootstrap scheme will be defined in terms of i.i.d.\ standard normally  distributed  random variables $\{R_i^{(\scs k)}\}_{i,k\in\N}$ which are independent of $\{(X_{t,T})_{t\in\Z}:T\in\N\}$. Further, let $m=m_T$ and $n=n_T$ denote two block length sequences satisfying one of the following two conditions. 

\begin{condition}\label{cond:bootstrap}~ \vspace{-.1cm}
	\begin{enumerate}
		\item[(B1)] 
		The block length $m=m_T \in\{1, \dots, T\}$ tends to infinity and satisfies $m=o(T)$ as $T\to\infty$.
		\item[(B2)] 
		The block length $n=n(T)\in\{1, \dots, T\}$ satisfies $m/n=o(1)$ and $m n^2=o(T^2)$ as $T\to\infty$.
	\end{enumerate}
\end{condition}
 
Next, let $K\in\N$ denote the number of bootstrap replications. For $k\in\{1, \dots,K\}$ and $h \in \{1, \dots, H\}$, define multiplier bootstrap approximations for 
\[
B_{h,T}(u, \tau_1, \tau_2) = \sqrt T \{ \hat {M}_{h,T}(u, \tau_1, \tau_2)- M_{h}(u, \tau_1, \tau_2)\}
\]
as
 \begin{align*}
 \hat{B}_{h,n,T}^{(k)}(u,\tau_1,\tau_2) 
 =
 \frac{1}{\sqrt{T}} \sum_{i=1}^{\lfloor uT\rfloor \wedge (T-h)}\frac{R_i^{(k)}}{\sqrt{m}} 
 \sum_{t=i}^{(i+m-1)\wedge (T-h)} \big\{ X_{t,T}(\tau_1) &X_{t+h,T}(\tau_2) \\
&  - \hat \mu_{t,h,n,T}(\tau_1,\tau_2) \big\},
 \end{align*}
 where
 \begin{align*} 
\hat \mu_{t,h,n,T}(\tau_1,\tau_2) =   \frac{1}{\tilde{n}_{t,h}}\sum_{j=\ubar{n}_t}^{\bar{n}_{t,h}}X_{t+j,T}(\tau_1)X_{t+j+h,T}(\tau_2)
 \end{align*}
 denotes the local empirical product moment of lag $h$ with
 \begin{equation*}
 \bar{n}_{t,h}=n\wedge (T-t-h), \quad \ubar{n}_t=-n\vee (1-t), \quad \tilde{n}_{t,h} = \bar{n}_{t,h}-\ubar{n}_t+1.
\end{equation*}
Note that for $n=T$ we obtain $\hat \mu_{t,h,T,T} = \hat \mu_{h,T}$ for all $t\in\{1, \dots, T\}$, where 
\[
\hat \mu_{h,T}(\tau_1,\tau_2) = \frac{1}{T-h}\sum_{t=1}^{T-h}X_{t,T}(\tau_1)X_{t+h,T}(\tau_2)
\]
denotes the global empirical product moment.
Let $\hat{\Bb}_{n,T}^{\scs (k)} = ( \hat{B}_{1,n,T}^{\scs (k)},\dots, \hat{B}_{H,n,T}^{\scs (k)})$ and $\hat{\Bb}_T = \sqrt{T}\big( B_{1,T}, \dots, B_{H,T}\big)$.
The following result shows that this multiplier bootstrap is consistent.

\begin{theorem}\label{theo:boot}
Suppose that Condition \ref{cond:fts} is met and let $\tilde{B}^{(1)},\tilde{B}^{(2)}, \dots$ denote independent copies of $\tilde{B}$. Fix $K,H\in\N$.
\begin{enumerate}
\item[(i)] If Condition~\ref{cond:bootstrap} (B1) and (B2) are met, then, as $T\to\infty$,
\[  
(\hat{\Bb}_T,\hat{\Bb}_{n,T}^{(1)},\dots,\hat{\Bb}_{n,T}^{(K)})
\convw
(\tilde{B},\tilde{B}^{(1)},\dots, \tilde{B}^{(K)}).
\]
\item[(ii)] If Condition~\ref{cond:bootstrap} (B1) is met and if $\cov(X_0^{\scs (0)}, X_h^{\scs (0)})=\cov(X_0^{\scs (w)}, X_h^{\scs (w)})$ for any $w\in[0,1]$ and $h\in\Z$,
then, as $T\to\infty$,
\[  
(\hat{\Bb}_T,\hat{\Bb}_{T,T}^{(1)},\dots,\hat{\Bb}_{T,T}^{(K)})
\convw
(\tilde{B},\tilde{B}^{(1)},\dots, \tilde{B}^{(K)}).
\]
 \end{enumerate}
 \end{theorem}

It is worthwhile to mention that the assumption on $\cov(X_0^{\scs (w)}, X_h^{\scs (w)})$ in Theorem~\ref{theo:boot}(ii) is met provided that $X_{t,T}=X_t$ for some stationary time series $(X_t)_{t\in \Z}$. In such a situation (for instance to be validated by a stationarity test in practice), using the bootstrap scheme with $n=T$ over the one with $n$ satisfying Condition~\ref{cond:bootstrap} (B2) typically results in better finite sample results, see Section~\ref{sec:sim} for more details.

Subsequently, we reconsider the  problem of testing  for  serial uncorrelation of  a locally stationary time series
using  classical  and relevant hypotheses. 
For  the sake of brevity, we only treat the hypotheses  $\bar H_0^{\scs (H)}$ and $\bar H_{0}^{\scs (H,\Delta)}$, which
are defined in    \eqref{eq:h02}  and  \eqref{hd3}, respectively  and  involve
multiple lags. For this purpose we consider the   following bootstrap approximations of the respective test statistics
\begin{align*}
\bar{\mathcal {S}}_{H,n,T}^{(k)} &= \max_{h=1}^H \| \hat{\Bb}_{h,n,T}^{(k)}  \|_{2,3}
\end{align*}
for the classical hypotheses and 
\begin{align*}
\bar{\mathcal {S}}_{H,n,T, \rel}^{(k)} &= \max_{h=1}^H \langle\hat M_{h,T},  \hat{\Bb}_{h,n,T}^{(k)} \rangle
\end{align*}
for the relevant hypotheses.   Finally, we propose  to reject the classical 
hypothesis   \eqref{eq:h02} whenever 
\begin{align} \label{classtest}
\bar p_{H,n,K,T}
=
\frac{1}{K}\sum_{k=1}^{K}\id\Big(\bar{\mathcal {S}}_{H,n,T}^{(k)} \geq  \bar{\mathcal S}_{H,T} \Big)
< \alpha ~.
\end{align}
Similarly, the relevant hypothesis  \eqref{hd3} is rejected whenever
\begin{align} \label{reltest}
\bar p_{H,n,K,T, \rel}
=
\frac{1}{K}\sum_{k=1}^{K}\id\Big(\bar{\mathcal {S}}_{H,n,T, \rel}^{(k)} \geq  \bar{\mathcal S}_{H,\Delta, T} \Big) < \alpha.
\end{align}

 \begin{corollary} \label{cor:test}
Fix $\alpha\in(0,1)$, suppose that Condition \ref{cond:fts} is met and let $K=K_T\to\infty$.  
\begin{enumerate}
\item[(i)] If Condition~\ref{cond:bootstrap} (B1) and (B2) hold, then the decision rule \eqref{classtest} 
defines  a consistent asymptotic level $\alpha$ test for  the classical hypotheses   \eqref{eq:h02}, that is
\[
\lim_{T \to \infty} \Prob(\bar p_{H,n,K,T} < \alpha) = \begin{cases} \alpha & \text{ under } \bar H_0^{(H),} \\
1 & \text{ else.}
\end{cases}
\]
Similarly, {for $\alpha<1/2$},  the decision rule  \eqref{reltest}  for  the relevant hypotheses    \eqref{hd3} satisfies
\begin{align} 
 \nonumber 
\hspace{2cm} \lim_{T\to\infty} \Prob(\bar p_{H,n,K,T, \rel} < \alpha) = 
0 & \quad \text{ if } \|M_h\|_{2,3} < \Delta_h \text{ for all } h\in\{1, \dots, H\}, \\
\label{hd4}
\limsup_{T \to \infty}\Prob(\bar p_{H,n,K,T, \rel} < \alpha) \le \alpha & \quad \text{ if } \bar H_0^{(H,\Delta)} \cap R \text{ is met,} \\
 \nonumber 
\lim_{T \to \infty}\Prob(\bar p_{H,n,K,T, \rel} < \alpha)=1 &\quad \text{ else},
\end{align}
where $R$ denotes the set of all models from the null hypothesis $ \bar H_0^{(H,\Delta)} $ for which $\|M_h\|_{2,3} = \Delta_h$  for some $h\in\{1, \dots, H\}$  and for which $\var(\langle M_h, \tilde B_h\rangle) >0$ for each such  $h$. In \eqref{hd4}, the value $\alpha$ is attained if $\|M_h\|_{2,3} = \Delta_h$ for all $h\in\{1, \dots, H\}$.

\item[(ii)]If Condition~\ref{cond:bootstrap} (B1) is met and if $\cov(X_0^{\scs (0)}, X_h^{\scs (0)})=\cov(X_0^{\scs (w)}, X_h^{\scs (w)})$ for any $w\in[0,1]$ and $h\in\N_0$, then the same assertions as in (i) are met for $n=T$.
\end{enumerate}
 \end{corollary}
 
The restriction to $\alpha<1/2$ for the test defined by \eqref{reltest}  is needed to make sure that the contribution from $\max_{h\in O_H} - \Delta_h \|\tilde B_h\|_{2,3}$ in Corollary~\ref{ConvM2} is negligible (see Section \ref{sec:proofs} for details).

\section{Monte Carlo Simulations} \label{sec:sim}

A large scale Monte Carlo simulation study was performed to analyse the finite-sample properties of the proposed tests. The major goal of the study was to analyse the level approximation and the power of the tests for hypotheses of the form $\bar{H}_0^{\scs (H)}$ and $\bar{H}_0^{\scs (H,\Delta)}$, with $H \in \{1,\dots, 4\}$.  Moreover, we also provide 
 a comparison with existing tests for white  noise / no serial correlation in the stationary setup, both for tests in the time domain \citep{kokoszka2017} and in the frequency domain \citep{zhang2016, bagchi2018,characiejus2020}. 

\subsection{Models}
We start by employing the same (stationary) models as in \cite{zhang2016} and \cite{bagchi2018}. 
In particular, for the null hypothesis of serial uncorrelation for any lag $h$, we consider: Model ($\mathrm{N}_1$), an i.i.d.\ sequence of Brownian motions; Model ($\mathrm{N}_2$),  an i.i.d.\ sequence of Brownian bridges; and Model ($\mathrm{N}_3$), data from a FARCH(1) process defined by
$$ X_{t}(\tau)=B_t(\tau) \sqrt{\tau + \int_0^1 c_\psi \exp\Big(\frac{\tau^2+\sigma^2}{2}\Big)X_{t-1}^2(\sigma)\diff \sigma}, $$
where $(B_t)_{t\in\Z}$ denotes an i.i.d.\ sequence of Brownian motions and $c_\psi = 0.3418$. Under the alternative, we consider the FAR(1) model given by
$$X_t
= \rho(X_{t-1}-\mu)+\eps_t,$$
where $\rho$ denotes an integral operator $\rho(f) = \int_0^1 K(\cdot,\sigma)f(\sigma)\diff \sigma,~f\in L^2([0,1])$, for a given kernel $K \in L^2([0,1]^2)$ and a sequence of centred, i.i.d. innovations $(\eps_t)_{t\in\Z}$ in $L^2([0,1])$. We consider the following choices for $K$ and $\eps_t$:
\begin{align*}
(\mathrm{A}_1) \quad &~ K(\tau,\sigma)=c_g \exp\big((\tau^2+\sigma^2)/2\big),~ \eps_t~\text{i.i.d. Brownian motions},\\
(\mathrm{A}_2) \quad &~ K(\tau,\sigma)=c_g \exp\big((\tau^2+\sigma^2)/2\big),~ \eps_t~\text{i.i.d. Brownian bridges}, \\
(\mathrm{A}_3) \quad &~ K(\tau,\sigma)=c_w \min(\tau,\sigma),~ \eps_t~\text{i.i.d. Brownian motions}, \\
(\mathrm{A}_4) \quad &~ K(\tau,\sigma)=c_w \min(\tau,\sigma),~ \eps_t~\text{i.i.d. Brownian bridges},
\end{align*}
where $c_g$ and $c_w$ are chosen such that the Hilbert-Schmidt norm of the $\rho$ is 0.3.

Note that the above models are stationary. Since our proposed methodology allows for smooth changes in the distribution of the underlying stochastic processes as well, we additionally consider the following heteroscedastic locally stationary models: 
\begin{align*}
(\mathrm{N}_4) \quad &~ X_{t,T} = \sigma(t/T) B_t, \\
(\mathrm{A}_5) \quad &~ X_{t,T} = \rho(X_{t-1,T})+\sigma(t/T) B_t,  \\
(\mathrm{A}_6)  \quad &~ X_{t,T} = \sigma(t/T)\rho(X_{t-1,T})+ B_t,
\end{align*}
where $(B_t)_{t\in\Z}$ denotes an i.i.d. sequence of Brownian motions, $\sigma(x)=x+1/2$ and $\rho$ is defined as in model ($\mathrm{A}_1$). For model ($\mathrm{N}_4$), the null hypothesis holds true, whereas the alternative is true for models ($\mathrm{A}_5$) and ($\mathrm{A}_6$).

\subsection{Details on the implementation}
For the comparison with the tests by \cite{zhang2016} and \cite{bagchi2018} (results in Table~\ref{table:zb}) and the evaluation of the finite-sample properties under non-stationarity (results in Tables~\ref{table:loc} and ~\ref{table:rel}), the data was simulated on an equidistant grid of size 1000 on the interval $[0,1]$. For the comparison with the tests by 
\cite{kokoszka2017} and \cite{characiejus2020} (results in Table~\ref{table:kc}), the size of the grid was chosen as 100 to accommodate the computational complexity of the tests. For the latter two tests, we relied on their implementation in the R-package \texttt{wwntests} by \cite{wwntests}.

For computational reasons, we reduced the dimension by projecting the generated data onto the subspace of $L^2([0,1])$ spanned by the first $D=17$ functions of the Fourier basis $\{\psi_n\}_{n\in\N_0}$, where, for $n\in\N$,
$$ \psi_0\equiv 1,\quad \psi_{2n-1}(\tau)=\sqrt{2}\sin(2\pi n\tau),\quad \psi_{2n}(\tau)=\sqrt{2}\cos(2\pi n\tau) $$
to calculate the proposed test statistic.

For the calculation of the bootstrap quantiles, we employed the data driven choice of the block length $m$ explained in \cite{Buecher2020}. %
In the context of stationary processes (models $(\mathrm{N}_1)$--$(\mathrm{N}_3)$ and $(\mathrm{A}_1)$--$(\mathrm{A}_4)$), it is natural to consider global estimators in the bootstrap procedure
and  we  chose the bandwidth $n=T$. In fact, preliminary simulations suggested that this choice of $n$ leads to better finite sample behavior. For the non-stationary models however, 
this choice is not reasonable and we  used local estimators in order to avoid a possible bias. In this setting, we chose the bandwidth $n=\lfloor T^{2/3}\rfloor$, satisfying Condition \ref{cond:bootstrap} (B2). The number of bootstrap replicates was chosen as $200$ and each model was simulated 1000 times.

{\scriptsize 
\begin{table}[t!]
		\begin{tabular}{c | rrrr  rr}
			 Model & $\bar{H}_0^{(1)}$ & $\bar{H}_0^{(2)}$ & $\bar{H}_0^{(3)}$ & $\bar{H}_0^{(4)}$ &  (B)  & (Z)  \\ 
			\hline \hline
			\addlinespace[.2cm]
		\multicolumn{7}{l}{\quad\textit{Panel A: $T=128$}} \\ \hline
			$(\mathrm{N}_1)$ & 7.3 & 6.3 & 5.9 & 5.6 & 1.8 & 4.2 \\ 
			$(\mathrm{N}_2)$ & 4.9 & 4.4 & 4.2 & 4.2 & 1.1 & 5.4 \\ 
			$(\mathrm{N}_3)$ & 5.4 & 4.5 & 4.2 & 3.8 & 4.7 & 5.9 \\ 
			$(\mathrm{A}_1)$ & 99.8 & 99.5 & 99.3 & 99.1 & 66.5 & 83.7 \\ 
			$(\mathrm{A}_2)$ & 98.4 & 97.9 & 97.1 & 96.5 & 51.7 & 83.1 \\ 
			$(\mathrm{A}_3)$ & 99.8 & 99.7 & 99.7 & 99.7 & 84.9 & 68.3 \\ 
			$(\mathrm{A}_4)$ & 91.8 & 88.5 & 85.7 & 82.7 & 37.0 & 65.8 \\ 
			 \hline\addlinespace[.2cm]
		\multicolumn{7}{l}{\quad\textit{Panel B: $T=256$}} \\ \hline
			$(\mathrm{N}_1)$ & 5.3 & 6.1 & 5.9 & 5.1 & 1.9 & 4.2 \\ 
			$(\mathrm{N}_2)$ & 4.4 & 5.3 & 4.8 & 4.2 & 1.4 & 4.8 \\ 
			$(\mathrm{N}_3)$ & 5.0 & 4.4 & 4.0 & 4.0 & 6.0 & 5.2 \\ 
			$(\mathrm{A}_1)$ & 100.0 & 100.0 & 100.0 & 100.0 & 91.5 & 99.2 \\ 
			$(\mathrm{A}_2)$ & 99.9 & 99.9 & 99.9 & 99.9 & 84.4 & 99.5 \\ 
			$(\mathrm{A}_3)$ & 100.0 & 100.0 & 100.0 & 100.0 & 99.1 & 98.2 \\ 
			$(\mathrm{A}_4)$ & 99.6 & 99.4 & 99.3 & 99.2 & 65.9 & 99.1 \\ 
			 \hline
		\end{tabular}   \medskip
		\hspace{.1cm}
		\begin{tabular}{  rrrr  rr}
			  $\bar{H}_0^{(1)}$ & $\bar{H}_0^{(2)}$ & $\bar{H}_0^{(3)}$ & $\bar{H}_0^{(4)}$ &  (B)  & (Z)  \\ 
			\hline \hline
			\addlinespace[.2cm]
		\multicolumn{6}{l}{\quad\textit{Panel C: $T=512$}} \\ \hline
			 6.1 & 5.4 & 6.0 & 5.1 & 2.8 & 4.7 \\ 
			 5.9 & 4.8 & 4.6 & 4.8 & 1.9 & 5.9 \\ 
			 4.3 & 5.0 & 4.7 & 4.0 & 6.3 & 4.9 \\ 
			 100.0 & 100.0 & 100.0 & 100.0 & 99.3 & 99.5 \\ 
			 100.0 & 100.0 & 100.0 & 100.0 & 98.3 & 99.8 \\ 
			 100.0 & 100.0 & 100.0 & 100.0 & 100.0 & 98.7 \\ 
			 100.0 & 100.0 & 100.0 & 100.0 & 90.4 & 100.0 \\
			\hline \addlinespace[.2cm]
		\multicolumn{6}{l}{\quad\textit{Panel D: $T=1024$}} \\ \hline
			 5.5 & 6.3 & 5.7 & 5.4 & 3.5 & 4.9 \\ 
			 5.5 & 5.7 & 5.3 & 5.4 & 3.5 & 5.1 \\ 
			 4.3 & 3.9 & 3.7 & 3.6 & 7.6 & 4.8 \\ 
			 100.0 & 100.0 & 100.0 & 100.0 & 100.0 & 100.0 \\ 
			 100.0 & 100.0 & 100.0 & 100.0 & 99.9 & 99.8 \\ 
			 100.0 & 100.0 & 100.0 & 100.0 & 100.0 & 100.0 \\ 
			 100.0 & 100.0 & 100.0 & 100.0 & 99.6 & 100.0 \\ \hline
		\end{tabular}   \medskip
		\caption{\it Empirical rejection rates  of  test  \eqref{classtest} 
for  the classical hypotheses  \eqref{eq:h02} in the case 
of stationary models, 
 for various values of the maximal lag $H$ in $\bar{H}_0^{\scs (H)}$.
The columns denoted by (B) and (Z) correspond to the tests of 
\cite{bagchi2018} and  \cite{zhang2016}, respectively.
}\label{table:zb}
	\end{table}
}

\subsection{Results for the classical hypotheses}

In the following, we denote by (B) and (Z) the tests proposed by \cite{bagchi2018} and \cite{zhang2016}, respectively. $(\mathrm{M}_H)$, $H\in\{1,2,3\}$, denotes the multiple-lag test at lag $H$ proposed by \cite{kokoszka2017}. Finally, $(\mathrm{Spec}_s)$ and $(\mathrm{Spec}_a)$ denote the spectral test as proposed by \cite{characiejus2020}, with static and adaptive bandwidth, respectively. The empirical rejection rates of test \eqref{classtest} for the stationary models  $(\mathrm{N}_1)$--$(\mathrm{N}_3)$ and $(\mathrm{A}_1)$--$(\mathrm{A}_4)$ are shown in Tables~\ref{table:zb} and \ref{table:kc}. 
We observe  that the level approximation  of the new test  \eqref{classtest} 
is very accurate for  all scenarios under consideration, and that the power is larger than for the competitors from the literature, in particular for small samples.
 A partial explanation for this observation consists in the fact that tests based in frequency  domain formulate the white noise hypothesis in terms of the spectral density operator and therefore implicitly consider the auto-covariance operators at any lag $h$. 
Although the power of test  \eqref{classtest} slightly decreases with increasing $H$, it decreases slower than the power of the multiple-lag time domain test by \cite{kokoszka2017}.
The type I errors of the tests $(\mathrm{Spec}_s)$ and $(\mathrm{Spec}_a)$ seem to exceed the level of $5\%$ for model $(\mathrm{N}_3)$. This difficulty might arise from  the fact that the data is uncorrelated but dependent. In contrast, the level approximation of the proposed tests seems to be more accurate.
 
The empirical rejection rates of test \eqref{classtest} for the locally stationary models $(\mathrm{N}_4)$, $(\mathrm{A}_5)$ and $(\mathrm{A}_6)$ are shown in  Table~\ref{table:loc}, 
for different sample sizes. We observe a reasonable approximation of the nominal level and high power under the non-stationary alternatives.

{\footnotesize
	\begin{table}{
			\begin{tabular}{l | rr rr  | rrrrr}
				Model & $\bar{H}_0^{(1)}$ & $\bar{H}_0^{(2)}$ & $\bar{H}_0^{(3)}$ & $\bar{H}_0^{(4)}$ & $(\mathrm{M}_1)$ & $(\mathrm{M}_2)$ & $(\mathrm{M}_3)$ & $(\mathrm{Spec}_{s})$ & $(\mathrm{Spec}_{a})$ \\ 
				\hline \hline
				\addlinespace[.2cm]
				\multicolumn{10}{l}{\quad\textit{Panel A: $T=100$}} \\ \hline
				$(\mathrm{N}_1)$ & 6.7 & 6.0 & 6.2 & 5.7  & 5.4 & 4.8 & 6.7 & 5.1 & 5.7 \\ 
				$(\mathrm{N}_2)$ & 5.1 & 4.5 & 4.8 & 4.4  & 3.0 & 3.5 & 3.7 & 4.3 & 5.0 \\ 
				$(\mathrm{N}_3)$ & 5.3 & 5.8 & 4.9 & 5.2  & 4.3 & 4.4 & 5.4 & 18.8 & 21.8 \\ 
				$(\mathrm{A}_1)$ & 97.5 & 96.3 & 95.4 & 94.6  & 96.6 & 92.3 & 88.0 & 100.0 & 99.7 \\ 
				$(\mathrm{A}_2)$ & 95.3 & 92.7 & 91.2 & 89.3  & 89.9 & 78.8 & 69.1 & 99.8 & 98.9 \\ 
				\hline
				\addlinespace[.2cm]
				\multicolumn{10}{l}{\quad\textit{Panel B: $T=200$}} \\ \hline
				$(\mathrm{N}_1)$ & 4.5 & 4.3 & 4.6 & 4.3 &  5.4 & 5.4 & 5.4 & 4.6 & 5.5 \\  
				$(\mathrm{N}_2)$ & 4.6 & 5.1 & 4.3 & 4.0 &  3.1 & 3.4 & 3.8 & 4.7 & 4.6 \\  
				$(\mathrm{N}_3)$ & 3.2 & 3.2 & 3.2 & 3.2 &  5.1 & 5.3 & 6.3 & 26.2 & 28.8 \\  
				$(\mathrm{A}_1)$ & 100.0 & 100.0 & 100.0 & 100.0 & 100.0 & 100.0 & 100.0 & 100.0 & 100.0 \\
				$(\mathrm{A}_2)$ & 100.0 & 99.9 & 99.9 & 99.8 & 100.0 & 100.0 & 98.7 & 100.0 & 100.0 \\  
				\hline
				\addlinespace[.2cm]
				\multicolumn{10}{l}{\quad\textit{Panel C: $T=300$}} \\ \hline
				$(\mathrm{N}_1)$ & 5.7 & 4.9 & 4.5 & 4.8 & 6.0 & 4.4 & 6.9 & 4.7 & 5.8 \\  
				$(\mathrm{N}_2)$ & 5.3 & 4.5 & 4.1 & 4.4 & 5.3 & 5.2 & 4.6 & 5.2 & 7.3 \\  
				$(\mathrm{N}_3)$ & 4.6 & 4.0 & 3.7 & 3.6 & 5.4 & 5.3 & 6.1 & 25.3 & 28.8 \\ 
				$(\mathrm{A}_1)$ & 100.0 & 100.0 & 100.0 & 100.0 & 100.0 & 100.0 & 100.0 & 100.0 & 100.0 \\ 
				$(\mathrm{A}_2)$ & 100.0 & 100.0 & 100.0 & 100.0 & 100.0 & 100.0 & 100.0 & 100.0 & 100.0
			\end{tabular}  \medskip 	
			\caption{\it Empirical rejection rates  of  test  \eqref{classtest} 
				for  the classical hypotheses  \eqref{eq:h02} in the case 
				of stationary models, 
				for various values of the maximal lag $H$ in $\bar{H}_0^{\scs (H)}$.
				The columns denoted by $(\mathrm{M}_i), i\in\{1,2,3\}$, and $(\mathrm{Spec}_i), i\in\{a,s\}$, correspond to the tests of \cite{kokoszka2017} and \cite{characiejus2020}, respectively.
			}\label{table:kc}
		}
	\end{table}
}

{\footnotesize
	\begin{table}[th!]
				\begin{tabular}{l | rr rr}
			Model & $\bar{H}_0^{(1)}$ & $\bar{H}_0^{(2)}$ & $\bar{H}_0^{(3)}$ & $\bar{H}_0^{(4)}$ \\ 
			\hline \hline
			\addlinespace[.2cm]
			\multicolumn{5}{l}{\quad\textit{Panel A: $T=128$}} \\ \hline
			$(\mathrm{N}_4)$ & 6.5 & 6.2 & 5.5 & 5.5 \\ 
			$(\mathrm{A}_5)$ & 99.0 & 98.6 & 98.6 & 98.6 \\ 
			$(\mathrm{A}_6)$ & 99.4 & 98.6 & 98.4 & 98.1 \\ 
			\hline
			\addlinespace[.2cm]
			\multicolumn{5}{l}{\quad\textit{Panel B: $T=256$}} \\ \hline
			$(\mathrm{N}_4)$ & 7.3 & 5.7 & 5.5 & 4.9 \\ 
			$(\mathrm{A}_5)$ & 100.0 & 100.0 & 100.0 & 100.0 \\ 
			$(\mathrm{A}_6)$ & 100.0 & 100.0 & 100.0 & 100.0 \\ 
			\hline
		\end{tabular}  \medskip
		\hspace{0.1cm}	\begin{tabular}{l | rr rr}
			Model & $\bar{H}_0^{(1)}$ & $\bar{H}_0^{(2)}$ & $\bar{H}_0^{(3)}$ & $\bar{H}_0^{(4)}$ \\ 
			\hline \hline
			\addlinespace[.2cm]
			\multicolumn{5}{l}{\quad\textit{Panel C: $T=512$}} \\ \hline
			$(\mathrm{N}_4)$ & 6.5 & 6.8 & 5.6 & 4.9 \\ 
			$(\mathrm{A}_5)$ & 100.0 & 100.0 & 100.0 & 100.0 \\ 
			$(\mathrm{A}_6)$ & 100.0 & 100.0 & 100.0 & 100.0 \\ 
			\hline
			\addlinespace[.2cm]
			\multicolumn{5}{l}{\quad\textit{Panel D: $T=1024$}} \\ \hline
			$(\mathrm{N}_4)$ & 6.9 & 6.5 & 6.5 & 5.0 \\ 
			$(\mathrm{A}_5)$ & 100.0 & 100.0 & 100.0 & 100.0 \\ 
			$(\mathrm{A}_6)$ & 100.0 & 100.0 & 100.0 & 100.0 \\ 
			\hline
		\end{tabular}  \medskip
				\caption{\it Empirical rejection rates  of  test  \eqref{classtest} 
			for  the classical hypotheses    \eqref{eq:h02} in the case 
			of locally stationary models, for various values for the maximal lag $H$ in $\bar{H}_0^{\scs (H)}$.
		}\label{table:loc}
	\end{table}
}
 
\subsection{Results for  relevant hypotheses}
We conclude this section with a brief discussion  of  the performance of the proposed test \eqref{reltest}
for the relevant hypotheses \eqref{hd3}. For this purpose we have calculated the quantities  $\|M_h\|_{2,3}$ for the models $(\mathrm{A}_1)$--$(\mathrm{A}_4)$
by a numerical simulation (specifically, we simulated $10,000$ time series of length $T=2,000$, projected them on a Fourier basis of dimension $D=101$, calculated for each time series the quantity $\|\hat{M}_{h,T}\|$, for $h\in\{1,\dots,4\}$, and used the respective means as an approximation for $\|M_h\|$). The results can be found in Table~\ref{tab:theoM}.
For the simulation experiment, we chose hypotheses corresponding to $\Delta=\Delta_{h,w} = w\|M_h\|_{2,3}$ with $w\in \{0.4+i/10: i=1,\dots, 11\}$ and $h=1,\dots, 4$, such that the null hypotheses are met for $w \ge 1$ and the alternative hypotheses are met for $w < 1$. The results can be found in Table \ref{table:rel}, where we omit the results for $H\in\{2,3\}$ since they are qualitatively similar to the cases $H\in\{1,4\}$. Again, we observe convincing level approximations and  good  power properties.

{\footnotesize
	\begin{table}[t!]
\begin{tabular}{c|c|c|c|c}
\text{Model} & $h=1$ & $h=2$ & $h=3$ & $h=4$ \\
\hline \hline
$(\mathrm{A}_1)$ & 0.1419 ~$\mathit{(5.00 \cdot 10^{-5})}$ & 0.0689 ~$\mathit{(3.93 \cdot 10^{-5})}$ & 0.0336 ~$\mathit{(3.41 \cdot 10^{-5})}$ & 0.0169 ~$\mathit{(2.94 \cdot 10^{-5})}$ \\	
$(\mathrm{A}_2)$ & 0.0283 ~$\mathit{(1.94 \cdot 10^{-6})}$ & 0.0138 ~$\mathit{(1.46 \cdot 10^{-6})}$ & 0.0069 ~$\mathit{(1.20 \cdot 10^{-6})}$ & 0.0037 ~$\mathit{(0.92 \cdot 10^{-6})}$\\
$(\mathrm{A}_3)$ & 0.1996 ~$\mathit{(2.51 \cdot 10^{-4})}$ & 0.1220  ~$\mathit{(2.02\cdot 10^{-4})}$ & 0.0755 ~$\mathit{(1.71 \cdot 10^{-4})}$ & 0.0468 ~$\mathit{(1.50\cdot 10^{-4})}$ \\
$(\mathrm{A}_4)$ & 0.0235 ~$\mathit{(4.40\cdot 10^{-6})}$ & 0.0117 ~$\mathit{(3.03\cdot 10^{-6})}$ & 0.0070 ~$\mathit{(2.15\cdot 10^{-6})}$ & 0.0048 ~$\mathit{(1.51 \cdot 10^{-6})} $
\end{tabular}\medskip
		\caption{\it Theoretical values of $\|M_h\|_{2,3}$, obtained by simulation. The numbers in brackets correspond to the empirical variance of the simulation.}\label{tab:theoM}
	\end{table}
}

{\footnotesize
	\begin{table}[t!]
		\begin{tabular}{ll| rrr rrr rrr rr}
		Model & $H\setminus w$ & 0.5 & 0.6 & 0.7 & 0.8 & 0.9 & \bf 1.0 & 1.1 & 1.2 & 1.3 & 1.4 & 1.5 \\ 
		\hline \hline 
		\addlinespace[.2cm]
		\multicolumn{7}{l}{\quad\textit{Panel A: $T=128$}} \\ \hline
		$(\mathrm{A}_1)$ & 1 & 58.1 & 39.4 & 25.2 & 13.2 & 8.1 & \bf 4.4 & 1.8 & 0.9 & 0.2 & 0.0 & 0.0 \\ 
		& 4 & 58.6 & 40.3 & 26.1 & 14.3 & 9.3 & \bf 5.3 & 2.6 & 1.3 & 0.6 & 0.3 & 0.2 \\ 
		$(\mathrm{A}_2)$ & 1 & 72.7 & 50.3 & 31.0 & 16.7 & 8.8 & \bf 3.9 & 2.2 & 1.3 & 0.3 & 0.2 & 0.2 \\ 
		& 4 & 75.0 & 53.7 & 35.3 & 20.2 & 11.7 & \bf 6.2 & 4.2 & 3.0 & 1.9 & 1.8 & 1.5 \\ 
		$(\mathrm{A}_3)$ & 1 & 57.9 & 40.1 & 22.6 & 13.2 & 6.9 & \bf 3.6 & 1.8 & 0.6 & 0.0 & 0.0 & 0.0 \\ 
		& 4 & 57.5 & 40.4 & 22.6 & 13.6 & 7.1 & \bf 3.9 & 2.0 & 0.7 & 0.0 & 0.0 & 0.0 \\ 
		$(\mathrm{A}_4)$ & 1 & 61.0 & 42.3 & 27.2 & 15.7 & 7.7 & \bf 4.7 & 3.0 & 1.4 & 0.7 & 0.2 & 0.2 \\ 
		& 4 & 67.7 & 49.2 & 34.2 & 21.6 & 12.4 & \bf 8.1 & 5.4 & 3.3 & 2.4 & 2.0 & 1.4 \\ 
		\hline
		\addlinespace[.2cm]
		\multicolumn{7}{l}{\quad\textit{Panel B: $T=256$}} \\ \hline
		$(\mathrm{A}_1)$ & 1 & 87.4 & 66.4 & 42.6 & 21.8 & 10.7 & \bf  4.4 & 1.4 & 0.2 & 0.0 & 0.0 & 0.0 \\ 
		& 4 & 87.3 & 66.3 & 42.8 & 22.0 & 10.9 & \bf 4.5 & 1.4 & 0.2 & 0.0 & 0.0 & 0.0 \\ 
		$(\mathrm{A}_2)$ & 1 & 91.7 & 75.7 & 51.9 & 28.4 & 12.5 & \bf 5.4 & 1.8 & 0.4 & 0.0 & 0.0 & 0.0 \\ 
		& 4 & 92.0 & 76.0 & 52.7 & 29.1 & 12.9 & \bf 5.8 & 2.2 & 0.5 & 0.1 & 0.1 & 0.1 \\ 
		$(\mathrm{A}_3)$ & 1 & 87.1 & 65.5 & 41.1 & 20.8 & 8.9 & \bf 3.6 & 0.9 & 0.2 & 0.0 & 0.0 & 0.0 \\ 
		& 4 & 87.1 & 65.6 & 41.2 & 20.8 & 8.9 & \bf 3.6 & 0.9 & 0.2 & 0.0 & 0.0 & 0.0 \\ 
		$(\mathrm{A}_4)$ & 1 & 82.8 & 64.0 & 41.9 & 22.8 & 12.6 & \bf 5.9 & 1.8 & 0.5 & 0.0 & 0.0 & 0.0 \\ 
		& 4 & 83.5 & 66.1 & 43.6 & 23.8 & 13.5 & \bf 6.7 & 2.4 & 0.8 & 0.1 & 0.1 & 0.1 \\ 
		\hline
		\addlinespace[.2cm]
		\multicolumn{7}{l}{\quad\textit{Panel C: $T=512$}} \\ \hline
		$(\mathrm{A}_1)$ & 1 & 97.4 & 88.6 & 68.2 & 39.8 & 16.6 & \bf 5.3 & 1.0 & 0.3 & 0.0 & 0.0 & 0.0 \\ 
		& 4 & 97.4 & 88.6 & 68.3 & 39.8 & 16.8 & \bf 5.5 & 1.2 & 0.5 & 0.0 & 0.0 & 0.0 \\ 
		$(\mathrm{A}_2)$ & 1 & 99.7 & 95.2 & 73.5 & 40.2 & 16.3 & \bf 4.4 & 1.0 & 0.2 & 0.0 & 0.0 & 0.0 \\ 
		& 4 & 99.7 & 95.2 & 73.3 & 40.4 & 16.5 & \bf 4.5 & 0.9 & 0.2 & 0.0 & 0.0 & 0.0 \\ 
		$(\mathrm{A}_3)$ & 1 & 98.1 & 88.2 & 64.1 & 37.0 & 13.4 & \bf 4.2 & 0.6 & 0.2 & 0.0 & 0.0 & 0.0 \\ 
		& 4 & 98.1 & 88.2 & 64.0 & 37.1 & 13.5 & \bf 4.2 & 0.7 & 0.3 & 0.0 & 0.0 & 0.0 \\ 
		$(\mathrm{A}_4)$ & 1 & 97.3 & 87.1 & 60.1 & 32.9 & 14.2 & \bf 4.1 & 1.0 & 0.2 & 0.0 & 0.0 & 0.0 \\ 
		& 4 & 97.3 & 87.3 & 60.0 & 33.1 & 14.6 & \bf 4.0 & 0.9 & 0.2 & 0.0 & 0.0 & 0.0 \\ 
		\hline
		\addlinespace[.2cm]
		\multicolumn{7}{l}{\quad\textit{Panel D: $T=1024$}} \\ \hline
		$(\mathrm{A}_1)$ & 1 & 100.0 & 99.7 & 89.6 & 60.3 & 23.8 & \bf  5.1 & 0.7 & 0.2 & 0.0 & 0.0 & 0.0 \\ 
		& 4 & 100.0 & 99.7 & 89.6 & 60.3 & 23.8 & \bf 5.1 & 0.8 & 0.2 & 0.0 & 0.0 & 0.0 \\ 
		$(\mathrm{A}_2)$ & 1 & 100.0 & 99.8 & 93.9 & 64.8 & 25.9 & \bf 5.0 & 0.5 & 0.0 & 0.0 & 0.0 & 0.0 \\ 
		& 4 & 100.0 & 99.8 & 93.9 & 64.7 & 26.0 & \bf 5.1 & 0.6 & 0.0 & 0.0 & 0.0 & 0.0 \\ 
		$(\mathrm{A}_3)$ & 1 & 100.0 & 99.6 & 87.0 & 53.8 & 18.6 & \bf 3.6 & 0.5 & 0.1 & 0.0 & 0.0 & 0.0 \\ 
		& 4 & 100.0 & 99.6 & 87.0 & 53.8 & 18.6 & \bf 3.6 & 0.6 & 0.1 & 0.0 & 0.0 & 0.0 \\ 
		$(\mathrm{A}_4)$ & 1 & 100.0 & 98.2 & 85.5 & 52.5 & 21.2 & \bf 5.0 & 0.6 & 0.0 & 0.0 & 0.0 & 0.0 \\ 
		& 4 & 100.0 & 98.2 & 85.5 & 52.6 & 21.2 & \bf 5.2 & 0.7 & 0.0 & 0.0 & 0.0 & 0.0 \\ 
		\hline
		\end{tabular}   \medskip
		\caption{\it Empirical rejection rates of the test 
		\eqref{reltest}   for  the relevant hypotheses    \eqref{hd3} in the case of 
	stationary models.} \label{table:rel} 
	\end{table}  
}

\section{Case Study} \label{sec:case}

{Functional data arises naturally when time series are recorded with a very high frequency. To illustrate the proposed methodology, we consider intraday prices of various stocks. More specifically, we consider prices over the time span from February 2016 to January 2020, where each observation corresponds to the intraday price at a given day. In particular, let $P_t(x_j)$, $t\in\{1,\dots, T\}, j\in\{1,\dots, m\}$ denote the price of a share, observed at time points $x_j$ at day $t$.  The lengths $T$ of the considered time series depend on the different stocks as for some days observations are missing. 
	
\cite{gabrys2010} define intradaily cumulative returns as
\[
R_t(x_j)= 100 \{ \log P_t(x_j) - \log P_t(x_1)   \} ,\quad  j\in\{1,\dots, m\},~ t\in\{1,\dots,T\}.
\] 
Throughout, we consider $R_t(\cdot)$ as an $L^2$-function. Some exemplary intradaily cumulative return curves are displayed in Figure~\ref{Figure1}.
The results of our testing procedure for detecting possible serial correlations can be found in Table \ref{Table6}, where we employed $K=1000$ bootstrap replicates and considered up to $H=4$ lags. The null hypotheses of serial correlation cannot be rejected at level $\alpha=0.05$, as the $p$-values clearly exceed $\alpha$. Thus, our results match the common assumption of uncorrelatedness in the literature.

	\begin{figure}[!t] 
	\centering
	\includegraphics[width=70mm]{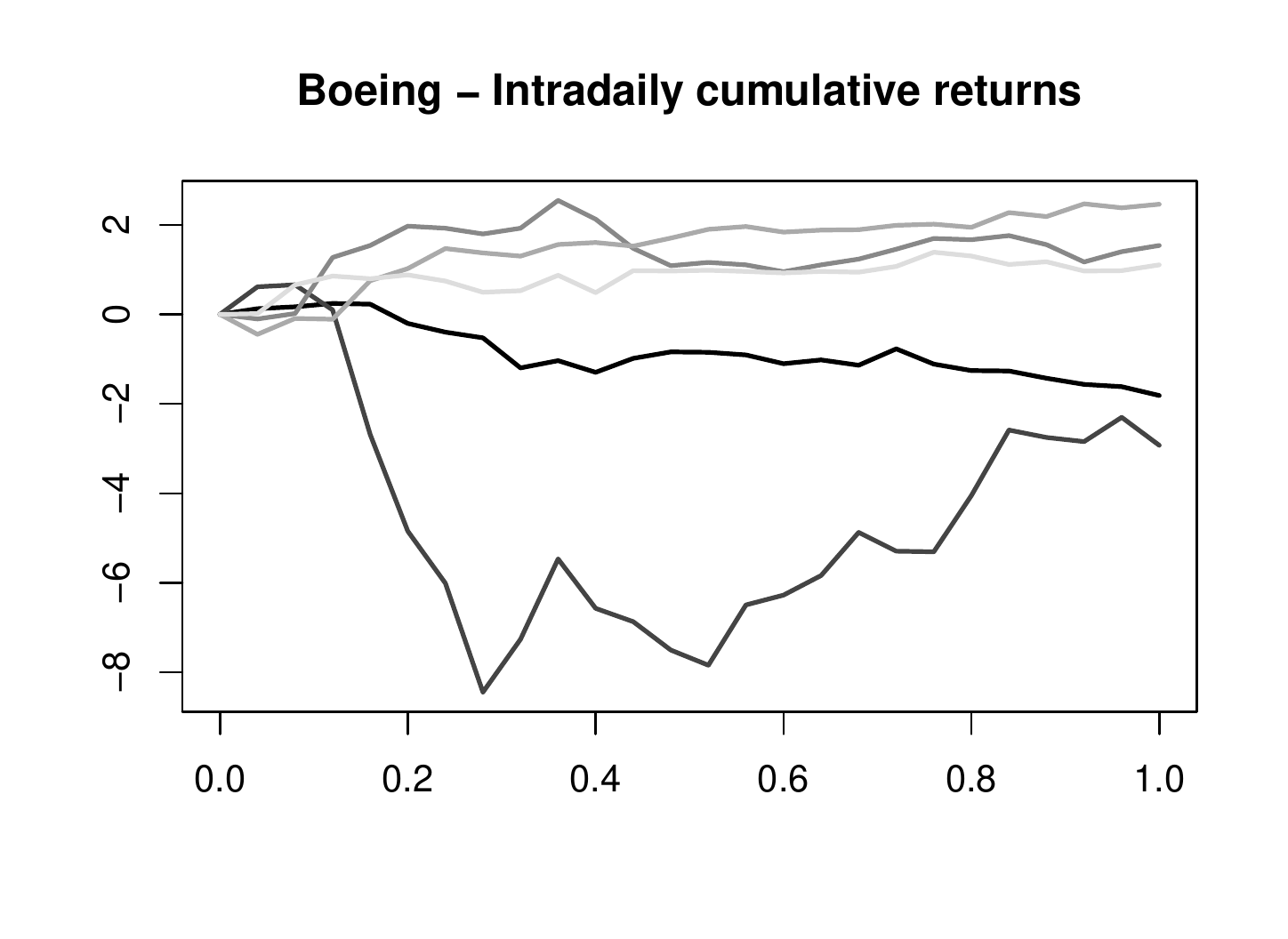}
	\includegraphics[width=70mm]{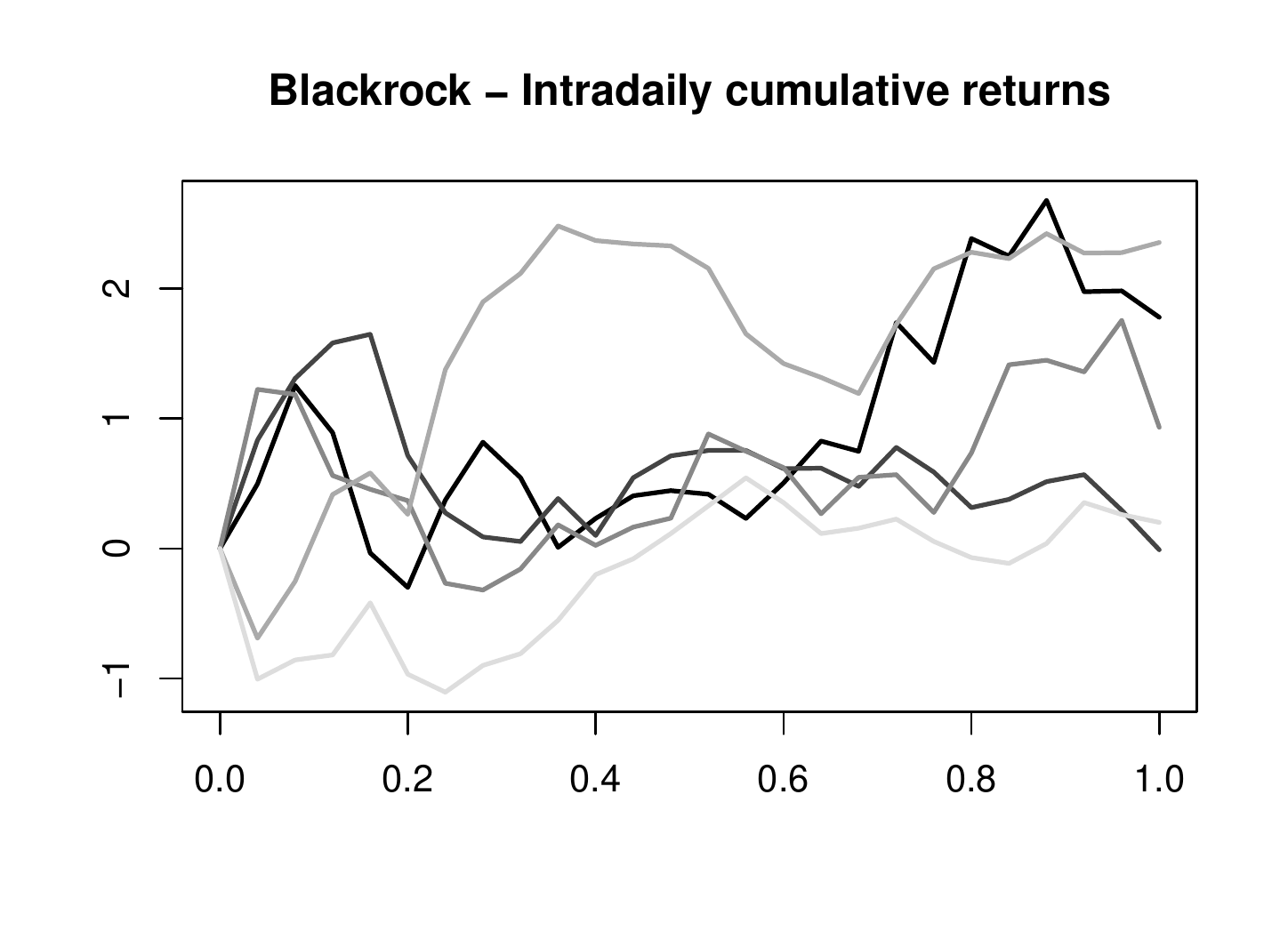}
	\vspace{-.6cm}
	\caption{ \it 
		Intradaily cumulative returns of Boeing and Blackrock from 8th to 12th of February 2016, where the x-axis corresponds to rescaled time and the y-axis denotes returns.}	\label{Figure1}
\end{figure}

{\footnotesize
	\begin{table}[t!]
	\begin{tabular}{l| rrrr|r}
		Stock & $\bar{H}_0^{(1)}$ & $\bar{H}_0^{(2)}$ & $\bar{H}_0^{(3)}$ & $\bar{H}_0^{(4)}$ & $T$ \\ 
		\hline \hline 
		Bank of America & 80.6 & 80.1 & 79.7 & 79.4 & 982 \\
		Blackrock & 98.4 & 98.3 & 98.3 & 98.3 & 822 \\
		Boeing & 82.3 & 81.9 & 81.4 & 81.1 & 984 \\
		Goldman Sachs & 74.5 & 73.9 & 73.6 & 73.3 & 990 \\
		JP Morgan & 93.3 & 93.2 & 93.0 & 92.9 & 982 \\
	\end{tabular}   \medskip
	\caption{\it $p$-values of the (combined) tests for the respective null hypotheses in percent.} \label{Table6} 
\end{table}  }
}

\section{Proofs} \label{sec:proofs}

\begin{proof}[Proof of Theorem~\ref{ConvM}]

We prove that  for any $H\in\N$ and as $T \to \infty$,
	\begin{align} \label{hd2}
	 \sqrt{T}\big( \hat{M}_{1,T}-M_1,\dots,\hat{M}_{H,T}-M_H  \big) \convw  \tilde{B} := (\tilde B_1,\dots,\tilde B_H) , 
	 \end{align}
	where $\tilde{B}$ denotes a centred Gaussian variable in $L^2([0,1]^3)^H$, with  covariance operator given by \eqref{eq:covOperator}.
The statement is then a  consequence of the continuous mapping theorem.
	
By Theorem 1 of \cite{Buecher2020}, the vector $\sqrt{T}(\hat{M}_{1,T}-\ex \hat{M}_{1,T},\dots, \hat{M}_{H,T}-\ex \hat{M}_{H,T})$ converges weakly to a vector of centred Gaussian variables $\tilde B = (\tilde B_1,\dots,\tilde B_H)$ in $L^2([0,1]^3)^H$. Thus, \eqref{hd2} follows from Slutsky's lemma, once we have shown that $\lim_{T \to\infty} \sqrt{T}\|\ex \hat{M}_{h,T}-M_h\|_{2,3}=0$ for any $h\in\N$.
For the latter purpose, invoke the triangle inequality to obtain
\begin{align*}
&\phantom{{}={}} \sqrt{T}\|\ex \hat{M}_{h,T}-M_h\|_{2,3}\\
&= \sqrt{T} \bigg(\int_0^1 \bigg\|\frac{1}{T}\sum_{t=1}^{\lfloor uT\rfloor \wedge (T-h)} \ex[X_{t,T}\otimes X_{t+h,T}]-\int_0^u \ex[X_0^{(w)}\otimes X_h^{(w)}]\diff w\bigg\|_{2,2}^2 \diff u \bigg)^{1/2}\\
&= 
\sqrt{T} \bigg(\int_0^1 \bigg\|\sum_{t=1}^{\lfloor uT\rfloor \wedge (T-h)} \int_{\frac{t-1}{T}}^{\frac{t}{T}}\ex[X_{t,T}\otimes X_{t+h,T}]-\ex[X_t^{(w)}\otimes X_{t+h}^{(w)}] \diff w \\
& \hspace{4.7cm} 
- 
\int_{T^{-1}\{\lfloor uT\rfloor \wedge (T-h)\}}^u  \ex[X_0^{(w)}\otimes X_h^{(w)}] \diff w\bigg\|_{2,2}^2 \diff u\bigg)^{1/2} \\
&\leq 
\sqrt{T} \bigg(\int_0^1 \bigg\{\sum_{t=1}^{\lfloor uT\rfloor \wedge (T-h)} \bigg\| \int_{\frac{t-1}{T}}^{\frac{t}{T}}\ex[X_{t,T}\otimes X_{t+h,T}-X_t^{(w)}\otimes X_{t+h}^{(w)}]\diff w\bigg\|_{2,2} \\
& \hspace{4.7cm} 
+ \bigg\| \int_{T^{-1}\{\lfloor uT\rfloor \wedge (T-h)\}}^u  \ex[X_0^{(w)}\otimes X_h^{(w)}] \diff w\bigg\|_{2,2}
\bigg\}^2 \diff u \bigg)^{1/2}.
\end{align*}
The integral from $T^{-1}\{\lfloor uT\rfloor \wedge (T-h)\}$ to $u$ at the right-hand side is of order $1/T$.
Further, by Jensen's inequality and local stationarity,
\begin{multline*}
\bigg\|\int_{\frac{t-1}{T}}^{\frac{t}{T}} \ex[X_{t,T}\otimes X_{t+h,T}-X_t^{(w)}\otimes X_{t+h}^{(w)}]\diff w\bigg\|_{2,2}\\
\leq \int_{\frac{t-1}{T}}^{\frac{t}{T}}\|\ex[X_{t,T}\otimes X_{t+h,T}-X_t^{(w)}\otimes X_{t+h}^{(w)}]\|_{2,2} \diff w \leq \frac{C}{T^2}
\end{multline*}
for some constant $C>0$.
Thus, it follows 
\begin{equation*} 
\sqrt{T}\|\ex \hat{M}_{h,T}-M_h\|_{2,3} = O(T^	{-1/2}),
\end{equation*} 
which completes the proof of  the theorem.
\end{proof}

\begin{proof}[Proof of Corollary~\ref{ConvM2}]
If $\|M_h\|_{2,3} = 0$ for some $h \in \{1, \dots H\}$, then $\sqrt{T}(\|\hat{M}_{h,T}\|_{2,3}-\|M_h\|_{2,3})\|\hat{M}_{h,T}\|_{2,3}$ converges to zero in probability by  Theorem~\ref{ConvM} and Slutsky's lemma. Hence, it is sufficient to assume that $\|M_h\|_{2,3} \ne 0$ for all $h \in \{1, \dots H\}$. We then obtain 
\begin{equation*} 
\sqrt{T}(\|\hat{M}_{h,T}\|_{2,3}-\|M_h\|_{2,3})_{h=1,\dots,H}\convw\bigg(\frac{\langle M_h,\tilde{B}_h\rangle}{\|M_h\|_{2,3}}\bigg)_{h=1,\dots,H}, 
\end{equation*}
from the functional delta method (Theorem 3.9.4  in \citealp{VanWel96}), applied to the functional in Proposition \ref{HadamardDiff} below. Apply Slutsky's lemma to conclude.
\end{proof}

\begin{proposition}
	\label{HadamardDiff}
	The function $\Phi:=\|\cdot\|_{2,3}$ from $L^2([0,1]^3)$ to $\R$ is Hadamard-differentia\-ble in any $M$ with $\|M\|_{2,3}>0$, with derivative $\Phi'_M(h)=\tfrac{\langle M,h\rangle}{\|M\|_{2,3}}$ in direction $h\in L^2([0,1]^3)$.
\end{proposition}
\begin{proof}
	For any sequences $h_n\to h$ with $h_n \in L^2([0,1]^3)$  and $t_n\to 0$ with $t_n \in \R\setminus \{0\}$, it holds
	\begin{align*}
	\frac{\|M+t_nh_n\|_{2,3}^2-\|M\|_{2,3}^2}{t_n}
	&= \frac{1}{t_n}\int_{[0,1]^3} 2 M(x)t_nh_n(x)+t_n^2h_n^2(x) \diff x\\
	&= \int_{[0,1]^3} 2M(x) h_n(x)\diff x +t_n\int_{[0,1]^3}h_n^2(x) \diff x,
	\end{align*}
	which converges to $2\int_{[0,1]^3}M(x)h(x) \diff x=2\langle M,h\rangle$. The square root function in $\R$ is Hadamard-differentiable at $x>0$ with derivative $(\sqrt{x})'=\frac{1}{2\sqrt{x}}$. By the chain rule for Hadamard-differentiable functions (Lemma 3.9.3 in \citealp{VanWel96}), the Hadamard-derivative of $\Phi$ is given by $\Phi'_M(h)=\tfrac{\langle M,h\rangle}{\|M\|_{2,3}}$.
\end{proof}

\begin{proof}[Proof of Theorem~\ref{theo:boot}]
	(i) can be deduced directly from Theorem 2 of \cite{Buecher2020}. For (ii) note that by Theorem C.3 of the supplementary material of the latter article, it holds $(\hat{\Bb}_T, \Bb_T^{(1)},\dots, \Bb_T^{(K)})\convw (\tilde{B},\tilde{B}^{(1)},\dots,\tilde{B}^{(K)})$, where $\Bb_T^{(k)} = (\tilde{B}_{T,1}^{(k)},\dots,\tilde{B}_{T,H}^{(k)})$ and
	\begin{multline*}
	\tilde{B}_{T,h}^{(k)}(u,\tau_1,\tau_2) \\
	= \frac{1}{\sqrt{T}}\sum_{i=1}^{\lfloor uT \rfloor \wedge (T-h)} \frac{R_i^{(k)}}{\sqrt{m}} \sum_{t=i}^{(i+m-1)\wedge(T-h)}X_{t,T}(\tau_1)X_{t+h,T}(\tau_2)-\ex[X_{t,T}(\tau_1)X_{t+h,T}(\tau_2)]. \end{multline*} 
	Note that for $u<1$ it holds $\lfloor uT\rfloor+m-1 \le T-h$, for any sufficiently large $T\in\N$. Thus, rewrite  
	\begin{multline*}
	\hat{B}_{h,T,T}^{(k)}(u,\tau_1,\tau_2)=\tilde{B}_{T,h}^{(k)}(u,\tau_1,\tau_2) \\+ \sqrt{\frac{m}{T}} \sum_{i=1}^{\lfloor uT \rfloor}R_i^{(k)} \bigg(\frac{1}{T-h}\sum_{t=1}^{T-h} \ex[X_{t,T}(\tau_1)X_{t+h,T}(\tau_2)]-X_{t,T}(\tau_1)X_{t+h,T}(\tau_2) \bigg)+\Oc_\pr\Big(\sqrt{\tfrac{m}{T}}\Big).
	\end{multline*}
	For the second term on the right-hand side of the latter display, it holds by independence of the random variables $R_i^{(k)}$,
	\begin{align*}
	&\phantom{{}={}} \ex\bigg\| \sqrt{\frac{m}{T}} \sum_{i=1}^{\lfloor \cdot  T \rfloor \wedge (T-h)}R_i^{(k)} \bigg(\frac{1}{T-h}\sum_{t=1}^{T-h} \ex[X_{t,T}\otimes X_{t+h,T}]-X_{t,T}\otimes X_{t+h,T} \bigg) \bigg\|_{2,3}^2\\
	&\leq \int_{[0,1]^2}m \ex\bigg[\bigg(\frac{1}{T-h}\sum_{t=1}^{T-h} \ex[X_{t,T}(\tau_1) X_{t+h,T}(\tau_2)]-X_{t,T}(\tau_1) X_{t+h,T}(\tau_2) \bigg)^2\bigg] \diff(\tau_1,\tau_2)\\
	&= \frac{m}{(T-h)^2} \sum_{t_1,t_2=1}^{T-h} \int_{[0,1]^2} \cov\big(X_{t_1,T}(\tau_1)X_{t_1+h,T}(\tau_2),X_{t_2,T}(\tau_1)X_{t_2+h,T}(\tau_2)\big) \diff(\tau_1,\tau_2),
	\end{align*}
	which is of order $O(m/T)$ by the same arguments as in the proof of Theorem 2 of \cite{Buecher2020}.
	Thus, $\hat{\Bb}_{T,T}^{(k)} = \Bb_T^{(k)} + O_\pr(\sqrt{m/T})$ and (ii) follows.
\end{proof}

\begin{proof}[Proof of Corollary~\ref{cor:test}]

The assertions for the null hypothesis $H_0^{\scs (H)}$ follow from Theorem~\ref{theo:boot} and Corollary 4.3 in \cite{BucKoj19}. The null hypothesis $H_0^{\scs (H, \Delta)}$ may be treated by similar arguments as in the last-named corollary, observing that the weak limit of $\bar {\mathcal S}_{H, \Delta,T}$ is stochastically bounded by $\max_{h=1}^H \langle M_h, \tilde B_h\rangle$ on the positive real line.
 The assertions regarding the alternative hypotheses follow from divergence to infinity of the test statistics and stochastic boundedness of the bootstrap statistics.  \end{proof}


\section*{Acknowledgements}
Financial support by the German Research Foundation is gratefully acknowledged. 
Collaborative Research Center “Statistical modelling of nonlinear dynamic processes” (SFB 823, Teilprojekt A1, A7 and C1) of the German Research Foundation (Deutsche Forschungsgemeinschaft, DFG) is gratefully acknowledged. Further, the research of H. Dette was partially supported by the German Research Foundation under Germany's Excellence Strategy - EXC 2092 CASA - 390781972.

The authors are grateful to V. Characiejus and G. Rice for providing their software and to N. Jumpertz for help with the data set.

\bibliographystyle{chicago}

\bibliography{bibliography}
\end{document}